\documentclass{ims9x6}

\makeindex

\makeatletter
\newcommand\chaptercontents{
{\global
\@topnum\z@
\@afterindentfalse
\if@twocolumn
\@restonecoltrue
\onecolumn
\else
\@restonecolfalse
\fi
\vspace*{10pt}
\noindent
{\small\bf Contents}\par
\vskip1em
\nobreak}
{\small
\@starttoc{toc}%
}\if@restonecol
\twocolumn
\fi}
\renewcommand*\l@section[2]{%
\ifnum \c@tocdepth >\z@
\addpenalty\@secpenalty
\setlength\@tempdima{2.5em}%
\begingroup
\parindent \z@
\rightskip
\@pnumwidth \parfillskip -\@pnumwidth
\leavevmode
\advance\leftskip\@tempdima
\hskip -\leftskip
#1\nobreak\leaderfill\nobreak
\hb@xt@\@pnumwidth{\hss #2}\par
\endgroup
\fi}
\renewcommand*\l@section{\@dottedtocline{1}{0.1em}{1.3em}}       
\renewcommand*\l@subsection{\@dottedtocline{2}{1.5em}{2em}}      
\renewcommand*\l@subsubsection{\@dottedtocline{3}{3.5em}{2.6em}} 

\makeatother

\usepackage{tikz} 
\usetikzlibrary{arrows}
\newtheorem{observation}[theorem]{Observation} 
\newcommand{\ZFC}{{\sf ZFC}}
\newcommand{\ZF}{{\sf ZF}}

\newcommand{\Inacc}{{\sf Inacc}}
\def\<#1>{\langle#1\rangle}
\newcommand{\Con}{\mathop{{\rm Con}}}

\newcommand{\Levy}{L\'{e}vy}
\newcommand{\satisfies}{\models}
\newcommand{\Ord}{\mathop{{\rm Ord}}}
\newcommand{\df}{\it}
\newcommand{\HOD}{{\rm HOD}}
\newcommand{\GBC}{{\sf GBC}}
\newcommand{\of}{\subseteq}
\newcommand{\MM}{{\sf MM}}
\newcommand{\PFA}{{\sf PFA}}
\newcommand{\MC}{{\sf MC}}
\newcommand{\SC}{{\sf SC}}
\newcommand{\LC}{{\sf LC}}

\newcommand{\elesub}{\prec}
\newcommand{\intersect}{\cap}
\newcommand{\smalllt}{\mathrel{\mathchoice{\raise2pt\hbox{$\scriptstyle<$}}{\raise1pt\hbox{$\scriptstyle<$}}{\raise0pt\hbox{$\scriptscriptstyle<$}}{\scriptscriptstyle<}}}

\newcommand{\proves}{\vdash}

\begin{document}

\chapter{A multiverse perspective on the axiom of constructiblity}

\markboth{J. D. Hamkins}{Multiverse perspective on $V=L$}

\author{Joel David Hamkins\footnote{This article expands on an argument that I made during my talk at the Asian Initiative for Infinity: Workshop on Infinity and Truth, held July 25--29, 2011 at the Institute for Mathematical Sciences, National University of Singapore. This work was undertaken during my subsequent visit at NYU in Summer and Fall, 2011, and completed when I returned to CUNY. My research has been supported in part by NSF grant DMS-0800762, PSC-CUNY grant 64732-00-42 and Simons Foundation grant 209252. Commentary concerning this paper can be made at http://jdh.hamkins.org/multiverse-perspective-on-constructibility.}
}

\address{
Visiting Professor of Philosophy, New York University\\
Professor of Mathematics, The City University of New York\\
           The Graduate Center \& College of Staten Island\\
 jhamkins@gc.cuny.edu, http://jdh.hamkins.org}

\begin{abstract}
 I shall argue that the commonly held $V\neq L$ via maximize position, which rejects the axiom of constructibility $V=L$ on the basis that it is restrictive, implicitly takes a stand in the pluralist debate in the philosophy of set theory by presuming an absolute background concept of ordinal. The argument appears to lose its force, in contrast, on an upwardly extensible concept of set, in light of the various facts showing that models of set theory generally have extensions to models of $V=L$ inside larger set-theoretic universes.
\end{abstract}

\vspace*{12pt}


\section{Introduction}

Set theorists often argue against the axiom of constructibility $V=L$ on the basis that it is restrictive. Some argue that we have no reason to think that every set should be
constructible, or as Shelah puts it, ``Why the hell should it be true?'' \cite{Shelah2003:LogicalDreams}. To suppose that every set is constructible is seen as an artificial limitation on set-theoretic possibility, and perhaps it is a mistaken principle generally to suppose that all structure is definable. Furthermore, although $V=L$ settles many set-theoretic questions, it seems so often to settle them in the `wrong' way, without the elegant smoothness and unifying vision of competing theories, such as the situation of descriptive set theory under $V=L$ in comparison with that under projective determinacy. As a result, the constructible universe becomes a pathological land of counterexamples. That is bad news, but it could be overlooked, in my opinion, were it not for the much worse related news that $V=L$ is inconsistent with all the strongest large cardinal axioms. The boundary between those large cardinals that can exist in $L$ and those that cannot is the threshold of set-theoretic strength, the entryway to the upper realm of infinity. Since the $V=L$ hypothesis is inconsistent with the largest large cardinals, it blocks access to that realm, and this is perceived as intolerably limiting. This incompatibility, I believe, rather than any issue of definabilism or descriptive set-theoretic consequentialism, is the source of the most strident end-of-the-line deal-breaking objections to the axiom of constructibility. Set theorists simply cannot accept an axiom that prevents access to their best and strongest theories, the large cardinal hypotheses, which encapsulate their dreams of what our set theory can achieve and express.

Maddy \cite{Maddy1988:BelievingTheAxiomsI,Maddy1988:BelievingTheAxiomsII} articulates the grounds that mathematicians often use in reaching this conclusion, mentioning especially the {\it maximize} maxim, saying ``the view that $V=L$ contradicts {\it maximize} is widespread,'' citing Drake, Moschovakis and Scott. Steel argues that ``$V=L$ is restrictive, in that adopting it limits the interpretative power of our language.'' He points out that the large cardinal set theorist can still understand the $V=L$ believer by means of the translation $\varphi\mapsto\varphi^L$, but ``there is no translation in the other direction'' and that ``adding $V=L$\ldots just prevents us from asking as many questions!'' \cite{Steel2004:SlidesGenericAbsolutenessAndTheContinuumProblem}. At bottom, the axiom of constructibility appears to be incompatible with strength in our set theory, and since we would like to study this strength, we reject the axiom.

Let me refer to this general line of reasoning as the {\it $V\neq L$ via maximize} argument. The thesis of this article is that the $V\neq L$ via maximize argument relies on a singularist as opposed to pluralist stand on the question whether there is an absolute background concept of ordinal, that is, whether the ordinals can be viewed as forming a unique completed totality. The argument, therefore, implicitly takes sides in the universe versus multiverse debate, and I shall argue that without that stand, the $V\neq L$ via maximize argument lacks force.

In \cite{Maddy1998:V=LAndMaximize}, Maddy gives the $V\neq L$ via maximize argument sturdier legs, fleshing out a more detailed mathematical account of it, based on a methodology of mathematical naturalism and using the idea that maximization involves realizing more isomorphism types. She begins with the `crude version' of the argument:
\begin{quote}
 The idea is simply this: there are things like $0^\sharp$
 that are not in L. And not only is $0^\sharp$ not in L; its existence implies the existence of an isomorphism type that is not realized by anything in L.\ldots
 So it seems that $\ZFC+{V{=}L}$ is restrictive because it rules out the extra isomorphism types available from $\ZFC+\exists 0^\sharp$. \cite[p. 142--143]{Maddy1998:V=LAndMaximize}
\end{quote}
For the full-blown argument, she introduces the concept of a `fair interpretation' of one theory in another and the idea of one theory maximizing over another, leading eventually to a proposal of what it means for a theory to be `restrictive' (see the details in section \ref{Section.CriticismOfMaddy}), showing that $\ZFC+{V{=}L}$ and other theories are restrictive, as expected, in that sense.

My thesis in this article is that the general line of the $V\neq L$ via maximize argument presumes that we have an absolute background concept of ordinal, that the ordinals build up to form an absolute completed totality. Of course, many set-theorists do take that stand, particularly set theorists in the California school. The view that the ordinals form an absolute completed totality follows, of course, from the closely related view that there is a unique absolute background concept of set, by which the sets accumulate to form the entire set-theoretic universe $V$, in which every set-theoretic assertion has a definitive final truth value. Martin essentially argues for the equivalence of these two commitments in his categoricity argument \cite{Martin2001:MultipleUniversesOfSetsAndIndeterminateTruthValues}, where he argues for the uniqueness of the set-theoretic universe, an argument that is a modern-day version of Zermelo's categoricity argument with strong parallels in Isaacson's \cite{Isaacson2008:TheRealityOfMathematicsAndTheCaseOfSetTheory}. Martin's argument is founded on the idea of an absolute unending well-ordered sequence of set-formation stages, an `Absolute Infinity' as with Cantor. Although Martin admits that `it is of course possible to have doubts about the sharpness of the concept of wellordering,'' \cite[p. 8]{Martin2001:MultipleUniversesOfSetsAndIndeterminateTruthValues}, his argument presumes that the concept is sharp, just as I claim the $V\neq L$ via maximize argument does.

Let me briefly summarize the position I am defending in this article, which I shall describe more fully section in \ref{Section.UpwardlyExtensibleConceptOfSet}. On the upwardly extensible concept of set, one holds that any given concept of set or set-theoretic universe may always be extended to a much better one, with more sets and larger ordinals. Perhaps the original universe even becomes a mere countable set in the extended universe. The `class of all ordinals', on this view, makes sense only relative to a particular set-theoretic universe, for there is no expectation that these extensions cohere or converge. This multiverse perspective resonates with or even follows from a higher-order version of the maximize principle, where we maximize not merely which sets exist, but also which set-theoretic universes exist. Specifically, it would be limiting for one set-theoretic universe to have all the ordinals, when we can imagine another universe looking upon it as countable. Maximize thereby leads us to expect that every set-theoretic universe should not only have extensions, but extremely rich extensions, satisfying extremely strong theories, with a full range of large cardinals. Meanwhile, I shall argue, the mathematical results of section \ref{Section.V=LCompatibleWithStrength} lead naturally to the additional conclusion that every set-theoretic universe should also have extensions satisfying $V=L$. In particular, even if we have very strong large cardinal axioms in our current set-theoretic universe $V$, there is a much larger universe $V^+$ in which the former universe $V$ is a countable transitive set and the axiom of constructibility holds. This perspective, by accommodating both large cardinals and $V=L$ in the multiverse, appears to dissolve the principal thrust of the $V\neq L$ via maximize argument. The idea that $V=L$ is permanently incompatible with large cardinals evaporates when we can have large cardinals and reattain $V=L$ in a larger domain. In this way, $V=L$ no longer seems restrictive, and the upward extensible concept of set reveals how large cardinals and other strong theories, as well as $V=L$, may all be pervasive as one moves up in the multiverse.

\section{Some new problems with Maddy's proposal}\label{Section.CriticismOfMaddy}

Although my main argument is concerned only with the general line of the $V\neq L$ via maximize position, rather than with Maddy's much more specific account of it in \cite{Maddy1998:V=LAndMaximize}, before continuing with my main agument I would nevertheless like to mention a few problems with that specific proposal. 

To quickly summarize the details, she defines that a theory $T$ {\df shows $\varphi$ is an inner model} if $T$ proves that $\varphi$ defines a transitive class satisfying every instance of an axiom of \ZFC, and either $T$ proves every ordinal is in the class, or $T$ proves that there is an inaccessible cardinal $\kappa$, such that every ordinal less than $\kappa$ is in the class. Next, $\varphi$ is a {\df fair interpretation} of $T$ in $T'$, where $T$ extends \ZFC, if $T'$ shows $\varphi$ is an inner model and $T'$ proves every axiom of $T$ for this inner model. A theory $T'$ {\df maximizes} over $T$, if there is a fair interpretation $\varphi$ of $T$ in $T'$, and $T'$ proves that this inner model is not everything (let's assume $T'$ includes \ZFC). The theory $T'$ {\df properly maximizes} over $T$ if it maximizes over $T$, but not conversely. The theory $T'$ {\df strongly maximizes} over $T$ if the theories contradict one another, $T'$ maximizes over $T$ and no consistent extension $T''$ of $T$ properly maximizes over $T'$. All of this culminates in her final proposal, which is to say that a theory $T$ is {\df restrictive} if and only if there is a consistent theory $T'$ that strongly maximizes over it.

Let me begin with a quibble concerning the syntactic form of her definition of `shows $\varphi$ is an inner model', which in effect requires $T$ to settle the question of whether the inner model is to contain all ordinals or instead merely all ordinals up to an inaccessible cardinal. That is, she requires that either $T$ proves that $\varphi$ is in the first case or that $T$ proves that $\varphi$ is in the second case, rather than the weaker requirement that $T$ prove merely that $\varphi$ is in one of the two cases (so the distinction is between $(T\proves A)\vee(T\proves B)$ and $T\proves A\vee B$). To illustrate how this distinction plays out in her proposal, consider the theory $\Inacc=\ZFC+$`there are unboundedly many inaccessible cardinals' and the theory $T=\ZFC+$`either there is a Mahlo cardinal or there are unboundedly many inaccessible cardinals in $L$.' (I shall assume without further remark that these large cardinal theories and the others I mention are consistent.) Every model of $T$ has an inner model of \Inacc, either by truncating at the Mahlo cardinal, if there is one, or by going to $L$, if there isn't. Thus, we seem to have inner models of the form Maddy desires. Unfortunately, however, this is not good enough, and I claim that \Inacc\ is actually {\it not} fairly interpreted in $T$. To see this,  notice first that $T$ does not prove the existence of an inaccessible cardinal, since we can force over any model of \Inacc\ by destroying all inaccessible cardinals and thereby produce a model of $T$ having no inaccessible cardinals.\footnote{First force `$\Ord$ is not Mahlo' by adding a closed unbounded class $C$ of non-inaccessible cardinals---this forcing adds no new sets---and then perform Easton forcing to ensure $2^\gamma=\delta^+$ whenever $\gamma$ is regular and $\delta$ is the next element of $C$.} Consequently, if $T$ shows $\varphi$ is an inner model, it cannot be because of the second clause, which requires $T$ to prove the existence of an inaccessible cardinal. Thus, $T$ must prove $\varphi$ holds of all ordinals. But notice also that $T$ does not prove that there are unboundedly many inaccessible cardinals in $L$, since by truncation we can easily have a Mahlo cardinal in $L$ with no inaccessible cardinals above it. So $T$ also cannot prove that $\varphi$ defines a proper class model of \Inacc. Thus, \Inacc\ is not fairly interpreted in $T$, even though we might have wished it to be. This issue can be addressed, of course, by modifying the definition of shows-an-inner-model to subsume the disjunction under the provability sign, that is, by requiring instead that $T$ prove the disjunction that either $\varphi$ holds of all ordinals or that it holds of all ordinals up to an inaccessible cardinal. But let me leave this issue; it does not affect my later comments.

My next objection is that the fairly-interpreted-in relation is not transitive, whereas our pre-reflective ideas for an interpreted-in relation would call for it to be transitive. That is, I claim that it can happen that a first theory has a fair interpretation in a second, which has a fair interpretation in a third, but the first theory has no fair interpretation in the third. Here is a specific example showing the lack of transitivity:
$$\begin{array}{rcl}
 R&=&\ZFC+{V{=}L}+\text{there is no inaccessible cardinal}\\
 S&=&\ZFC+{V{=}L}+\text{there is an inaccessible cardinal}\\
 T&=&\ZFC+\omega_1\text{ is inaccessible in }L\\
\end{array}$$
The reader may easily verify that $R$ has a fair interpretation in $S$ by truncating the universe at the first inaccessible cardinal, and $S$ has a fair interpretation in $T$ by going to $L$. Furthermore, every model of $S$ has forcing extensions satisfying $T$, by the \Levy\ collapse. Meanwhile, I claim that $R$ has no fair interpretation in $T$. The reason is that $T$ is consistent with the lack of inaccessible cardinals, and so if $T$ shows $\varphi$ is an inner model, then in any model of $T$ having no inaccessible cardinals, this inner model must contain all the ordinals. In this case, in order for it to have $R^\varphi$, the inner model must be all of $L$, which according to $T$ has an inaccessible cardinal, and therefore doesn't satisfy $R$ after all. So $R$ is not fairly interpreted in $T$. The reader may construct many similar examples of intransitivity. The essence here is that the first theory is fairly interpreted in the second only by truncating, and the second is fairly interpreted in the third only by going to an inner model containing all the ordinals, but there is no way to interpret the first in the third except by doing both, which is not allowed in the definition if the truncation point is inaccessible only in the inner model and not in the larger universe.

The same example shows that the maximizing-over relation also is not transitive, since $T$ maximizes over $S$ and $S$ maximizes over $R$, by the fair interpretations mentioned above (note that these theories are mutually exclusive), but $T$ does not maximize over $R$, since $R$ has no fair interpretation in $T$. Similarly, the reader may verify that the example shows that the properly-maximizes-over and the strongly-maximizes-over relations also are not transitive.

Let me turn now to give a few additional examples of what Maddy calls a `false positive,' a theory deemed formally restrictive, which we do not find intuitively to be restrictive. As I see it, the main purpose of \cite{Maddy1998:V=LAndMaximize} is to give precise mathematical substance to the intuitive idea that some set theories seem restrictive in a way that others do not. We view $V=L$ and `there is a largest inaccessible cardinal' as limiting, while `there are unboundedly many inaccessible cardinals' seems open-ended and unrestrictive. Maddy presents some false positives, including an example of Steel's showing that $\ZFC+$`there is a measurable cardinal' is restrictive because it is strongly maximized by the theory $\ZFC+0^\dagger\text{ exists}+\forall\alpha<\omega_1\ L_\alpha[0^\dagger]\not\satisfies\ZFC$. L\"owe points out that ``this example can be generalized to at least every interesting theory in large cardinal form extending \ZFC. Thus, most theories are restrictive in a formal sense,'' \cite{Loewe2001:AFirstGlanceAtNonrestrictiveness} and he shows in \cite{Loewe2003:ASecondGlanceAtNonrestrictiveness} that \ZFC\ itself is formally restrictive because it is maximized by the theory $\ZF+$`every uncountable cardinal is singular'.

I would like to present examples of a different type, which involve what I believe to be more attractive maximizing theories that seem to avoid the counterarguments that have been made to the previous examples of false positives. First, consider again the theory \Inacc, asserting $\ZFC+$`there are unboundedly many inaccessible cardinals', a theory Maddy wants to regard as not restrictive. Let $T$ be the theory asserting $\ZFC+$`there are unboundedly many inaccessible cardinals in $L$, but no worldly cardinals in $V$.' A cardinal $\kappa$ is {\df worldly} when $V_\kappa\satisfies\ZFC$. Worldliness is a weakening of inaccessibility, since every inaccessible cardinal is worldly and in fact a limit of worldly cardinals; but meanwhile, worldly cardinals need not be regular, and the regular worldly cardinals are exactly the inaccessible cardinals. The worldly cardinals often serve as a substitute for inaccessible cardinals, allowing one to weaken the large cardinal commitment of a hypothesis. For example, one may carry out most uses of the Grothendieck universe axiom in category theory by using mere worldly cardinals in place of inaccessible cardinals. The theory $T$ is equiconsistent with \Inacc, since every model of \Inacc\ has a class forcing extension of $T$.\footnote{One first adds a closed unbounded class $C$ of cardinals containing no worldly cardinals (this forcing adds no new sets), and then performs Easton forcing so as to ensure that $2^\gamma=\delta^+$, where $\gamma$ is regular and $\delta$ is the next element of $C$. The result is a model of $T$, since all worldly cardinals have been killed off.} The theory \Inacc\ has a fair interpretation in $T$, by going to $L$, and as a result, $T$ maximizes over \Inacc. Meanwhile, I claim that no strengthening of \Inacc\ properly maximizes over $T$. To see this, suppose that $\Inacc^+$ contains \Inacc\ and shows $\varphi$ is an inner model $M$ satisfying $T$. If $M$ contains all the ordinals, then since \Inacc\ proves that the inaccessible cardinals are unbounded, $M$ would have to contain all those inaccessible cardinals, which would remain inaccessible in $M$ since inaccessibility is downward absolute, and therefore violate the claim of $T$ that there are no worldly cardinals. So by the definition of fair interpretation, therefore, $M$ would have to contain all the ordinals up to an inaccessible cardinal $\kappa$. But in this case, a Loweheim-Skolem argument shows that there is a closed unbounded set of $\gamma<\kappa$ with $V_\gamma^M\elesub V_\kappa^M$, and all such $\gamma$ would be worldly cardinals in $M$, violating $T$. Thus, \Inacc\ is strongly maximized by $T$, and so \Inacc\ is restrictive.

Let me improve the example to make it more attractive, provided that we read Maddy's definition of `fair interpretation' in a way that I believe she may have intended. The issue is that although Maddy refers to `truncation\ldots at inaccessible levels' and her definition is typically described by others using that phrase, nevertheless the particular way that she wrote her definition does not actually ensure that the truncation occurs at an inaccessible level. Specifically, in the truncation case, she writes that $T$ should prove that there is an inaccessible cardinal $\kappa$ for which $\forall\alpha(\alpha<\kappa\to\varphi(\alpha))$. But should this implication be a biconditional? Otherwise, of course, nothing prevents $\varphi$ from continuing past $\kappa$, and the definition would be more accurately described as `truncation at, or somewhere above, an inaccessible cardinal'. If one wants to allow truncation at non-inaccessible cardinals, why should we bother to insist that the height should exceed some inaccessible cardinal? Replacing this implication with a biconditional would indeed ensure that when the inner model arises by truncation, it does so by truncating at an inaccessible cardinal level. So let us modify the reading of `fair interpretation' so that truncation, if it occurs, does so at an inaccessible cardinal level. In this case, consider the theory \Inacc\ as before, and let $\MC^*$ be the theory $\ZFC+$`there is a measurable cardinal with no worldly cardinals above it'. By truncating at a measurable cardinal, we produce a model of \Inacc, and so $\MC^*$ offers a fair interpretation of \Inacc, and consequently $\MC^*$ maximizes over \Inacc. But no consistent strengthening of \Inacc\ can maximize over $\MC^*$, since if $V\satisfies\Inacc$ and $W$ is an inner model of $V$ satisfying $\MC^*$, then $W$ cannot contain all the ordinals of $V$, since the inaccessible cardinals would be worldly in $W$, and neither can the height of $W$ be inaccessible in $V$, since if $\kappa=W\intersect\Ord$ is inaccessible in $V$, then by a Lowenheim-Skolem argument there must be a closed unbounded set of $\gamma<\kappa$ such that $W_\gamma\elesub W$, and this will cause unboundedly many worldly cardinals in $W$, contrary to $\MC^*$. Thus, on the modified definition of fair interpretation, we conclude that $\MC^*$ strongly maximizes over \Inacc, and so \Inacc\ is restricted.

One may construct similar examples using the theory $\ZFC+$`there is a proper class of measurable cardinals', which is strongly maximized by $\SC^*=\ZFC+$`there is a supercompact cardinal with no worldly cardinals above it'. Truncating at the supercompact cardinal produces a model of the former theory, but no strengthening of the former theory can show $\SC^*$ in an inner model, since the unboundedly many measurable cardinals of the former theory prevent the showing of any proper class model of $\SC^*$, and the eventual lack of worldly cardinals in $\SC^*$ prevents it from being shown in any truncation at an inaccessible level of any model of \ZFC. A general format for these examples would be $\ZFC+$`there is a proper class of large cardinals of type \LC' and $T=\ZFC+$`there is an inaccessible limit of \LC\ cardinals, with no worldly cardinals above.' Such examples work for any large cardinal notion \LC\ that implies worldliness, is absolute to truncations at inaccessible levels and is consistent with a lack of worldly cardinals above. Almost all (but not all) of the standard large cardinal notions have these features.

Maddy has rejected some of the false positives on the grounds that the strongly maximizing theory involved is a `dud' theory, such as $\ZFC+\neg\Con(\ZFC)$. Are the theories above, $\MC^*$ and $\SC^*$, duds in this sense? It seems hard to argue that they are. For various reasons, set theorists often consider models of set theory with largest instances of large cardinals and no large cardinals above, often obtaining such models by truncation, in order to facilitate certain constructions. Indeed, the idea of truncating the universe at an inaccessible cardinal level lies at the heart of Maddy's definitions. But much of the value of that idea is already obtained when one truncates at the worldly cardinals instead. The theory $\MC^*$ can be obtained from any model of measurable cardinal by truncating at the least worldly cardinal above it, if there is one. But moreover, one needn't truncate at all: one can force $\MC^*$ over any model with a measurable cardinal, by very mild forcing. First, add a closed unbounded class $C$ of cardinals containing no worldly cardinal, and then perform Easton forcing so as to ensure in the forcing extension that $2^\gamma=\delta^+$, whenever $\gamma$ is regular and $\delta$ is the next element of $C$ above $\gamma$. The point is that this forcing will ensure that the continuum function $\gamma\mapsto 2^\gamma$ jumps over the former worldly cardinals, and so they will no longer be worldly (and no new worldly cardinals are created). If one starts this forcing above a measurable cardinal $\kappa$, then one preserves that measurable cardinal while killing all the worldly cardinals above it. (In the case of $\SC^*$, one should first make the supercompact cardinal Laver indestructible.) Because we can obtain $\MC^*$ and $\SC^*$ by moving from a large cardinal model to a forcing extension, where all the previous context and strength seems still available, these theories do not seem to be duds in any obvious way. Nevertheless, the theories $\MC^*$ and $\SC^*$ are restrictive, of course, in the intuitive sense that Maddy's project is concerned with. But to object that these theories are duds on the grounds that they are restrictive would be to give up the entire project; the point was to give precise substance to our notion of `restrictive', and it would beg the question to define that a theory is restrictive if it is strongly maximized by a theory that is not `restrictive.'


\section{Several ways in which V=L is compatible with strength}\label{Section.V=LCompatibleWithStrength}

In order to support my main thesis, I would like next to survey a series of mathematical results, most of them a part of set-theoretic folklore, which reveal various senses in which the axiom of constructibility $V=L$ is compatible with strength in set theory, particularly if one has in mind the possibility of moving from one universe of set theory to a much larger one.

First, there is the easy observation, expressed in observation \ref{Observation.VandLAgreeOnConsistency}, that $L$ and $V$ satisfy the same consistency assertions. For any
constructible theory $T$ in any language---and by a `constructible' theory I mean just that $T\in L$, which is true of any c.e.~theory, such as \ZFC\ plus any of the usual large
cardinal hypotheses---the constructible universe $L$ and $V$ agree on the consistency of $T$ because they have exactly the same proofs from $T$. It follows from this, by the
completeness theorem, that they also have models of exactly the same constructible theories.

\begin{observation}\label{Observation.VandLAgreeOnConsistency} 
 The constructible universe $L$ and $V$ agree on the consistency of any constructible theory. They have models of the same constructible theories.
\end{observation}

What this easy fact shows, therefore, is that while asserting $V=L$ we may continue to make all the same consistency assertions, such as $\Con(\ZFC+\exists\text{ measurable
cardinal})$, with exactly the same confidence that we might hope to do so in $V$, and we correspondingly find models of our favorite strong theories inside $L$. Perhaps a skeptic
worries that those models in $L$ are somehow defective? Perhaps we find only ill-founded models of our strong theory in $L$? Not at all, in light of the following theorem, a fact
that I found eye-opening when I first came to know it years ago.

\begin{theorem}\label{Theorem.VandLhaveSameTransitiveModels}
 The constructible universe $L$ and $V$ have transitive models of exactly the same constructible theories in the language of set theory.
\end{theorem}

\begin{proof}
The assertion that a given theory $T$ has a transitive model has complexity $\Sigma^1_2(T)$, in the form ``there is a real coding a well founded structure satisfying $T$,'' and so it is absolute between $L$ and $V$ by the Shoenfield absoluteness theorem, provided the theory itself is in $L$.
\end{proof}

Consequently, one can have transitive models of extremely strong large cardinal theories without ever leaving $L$. For example, if there is a transitive model of the theory
$\ZFC+$``there is a proper class of Woodin cardinals,'' then there is such a transitive model inside $L$. The theorem has the following interesting consequence.

\begin{corollary}\label{Corollary.LandVSameSigma_1theory}(Levy-Shoenfield absoluteness theorem)
 In particular, $L$ and $V$ satisfy the same $\Sigma_1$ sentences, with parameters hereditarily countable in $L$. Indeed, $L_{\omega_1^L}$ and $V$ satisfy the same such sentences.
\end{corollary}

\begin{proof}
Since $L$ is a transitive class, it follows that $L$ is a $\Delta_0$-elementary substructure of $V$, and so $\Sigma_1$ truth easily goes upward from $L$ to $V$. Conversely, suppose
$V$ satisfies $\exists x\,\varphi(x,z)$, where $\varphi$ is $\Delta_0$ and $z$ is hereditarily countable in $L$. Thus, $V$ has a transitive model of the theory $\exists
x\varphi(x,z)$, together with the atomic diagram of the transitive closure $z$ and a bijection of it to $\omega$. By observation \ref{Observation.VandLAgreeOnConsistency}, it follows
that $L$ has such a model as well. But a transitive model of this theory in $L$ implies that there really is an $x\in L$ with $\varphi(x,z)$, as desired. Since the witness is
countable in $L$, we find the witness in $L_{\omega_1^L}$.
\end{proof}

One may conversely supply a direct proof of corollary \ref{Corollary.LandVSameSigma_1theory} via the Shoenfield absoluteness theorem and then view theorem
\ref{Theorem.VandLhaveSameTransitiveModels} as the consequence, because the assertion that there is a transitive model of a given theory in $L$ is $\Sigma_1$ assertion about that
theory.

I should like now to go further. Not only do $L$ and $V$ have transitive models of the same strong theories, but what is more, any given model of set theory can, in principle, be
continued to a model of $V=L$. Consider first the case of a countable transitive model $\<M,{\in}>$.

\begin{theorem}\label{Theorem.EveryCTMisInsideOmegaModelOfV=L}
 Every countable transitive set is a countable transitive set in the well-founded part of an $\omega$-model of $V=L$.
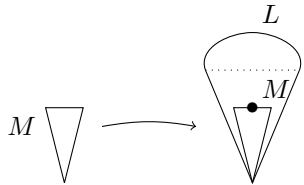
$$\begin{tikzpicture}[scale=.25]
 \draw (-10,0) --(-9,4) --(-11,4) --(-10,0);
 \draw (0,0) --(1,4) --(-1,4) --(0,0);
 \draw[->] (-8,3) to [out=10,in=170](-3,3);
 \draw (-2.5,6) --(0,0) --(2.5,6);
 \draw (2.5,6) to [out=70,in=0] (0,8) to [out=180,in=110](-2.5,6);
 \draw[dotted] (-2.5,6) --(2.5,6);
 \node at (0,4) (M) {$\bullet$};
 \node[anchor=north east] at (-11,4) {$M$};
 \node[anchor=south west] at (M) {$M$};
 \node[anchor=south west] at (0,8) {$L$};
\end{tikzpicture}$$
\end{theorem}

\begin{proof} The statement is true inside $L$, since every countable transitive set of $L$ is an element of some countable $L_\alpha$, which is transitive and satisfies $V=L$.
Further, the complexity of the assertion is $\Pi^1_2$, since it asserts that for every countable transitive set, there is another countable object satisfying a certain arithmetic
property in relation to it. Consequently, by the Shoenfield absoluteness theorem, the statement is true.
\end{proof}

Thus, every countable transitive set has an end-extension to a model of $V=L$ in which it is a set. In particular, if we have a countable transitive model $\<M,{\in}>\satisfies\ZFC$, and perhaps this is a model of some very strong large cardinal theory, such as a proper class of supercompact cardinals, then nevertheless there is a model $\<N,\in^N>\satisfies V=L$ which has $M$ as an element, in such a way that the membership relation of $\in^N$ agrees with $\in$ on the members of $M$. This implies that the ordinals of $N$ are well-founded at
least to the height of $M$, and so not only is $N$ an $\omega$-model, but it is an $\xi$-model where $\xi=\Ord^M$, and we may assume that the membership relation $\in^N$ of $N$ is
the standard relation $\in$ for sets of rank up to and far exceeding $\xi$. Furthermore, we may additionally arrange that the model satisfies $\ZFC^-$, or any desired finite fragment of
\ZFC, since this additional requirement is achievable in $L$ and the assertion that it is met still has complexity $\Pi^1_2$. If there are arbitrarily large $\lambda<\omega_1^L$ with $L_\lambda\satisfies\ZFC$, a hypothesis that follows from the existence of a single inaccessible cardinal (or merely from an uncountable transitive model of \ZF), then one can
similarly obtain \ZFC\ in the desired end-extension.

A model of set theory is {\df pointwise definable} if every object in the model is definable there without parameters. This implies $V=\HOD$, since in fact no ordinal parameters are
required, and one should view it as an extremely strong form of $V=\HOD$, although the pointwise definability property, since it implies that the model is countable, is not
first-order expressible. The main theorem of \cite{HamkinsLinetskyReitz:PointwiseDefinableModelsOfSetTheory} is that every countable model of \ZFC\ (and similarly for \GBC) has a
class forcing extension that is pointwise definable.

\begin{theorem}\label{Theorem.EveryCTSinsidePDMofV=L}
 If there are arbitrarily large $\lambda<\omega_1^L$ with $L_\lambda\satisfies\ZFC$, then every countable transitive set $M$ is a countable transitive set inside a structure $M^+$
 that is a pointwise-definable model of $\ZFC+{V{=}L}$, and $M^+$ is well founded as high in the countable ordinals as desired.
\end{theorem}

\begin{proof}
See \cite[Corollary 10]{HamkinsLinetskyReitz:PointwiseDefinableModelsOfSetTheory} for further details. First, note that every real $z$ in $L$ is in a pointwise definable $L_\alpha$,
since otherwise, the $L$-least counterexample $z$ would be definable in $L_{\omega_1}$ and hence in the the Skolem hull of $\emptyset$ in $L_{\omega_1}$, which collapses to a
pointwise definable $L_\alpha$ in which $z$ is definable, a contradiction. For any such $\alpha$, let $L_\lambda\satisfies\ZFC$ have exactly $\alpha$ many smaller $L_\beta$
satisfying \ZFC, and so $\alpha$ and hence also $z$ is definable in $L_\lambda$, whose Skolem hull of $\emptyset$ therefore collapses to a pointwise definable model of $\ZFC+{V{=}L}$
containing $z$. So the conclusion of the theorem is true in $L$. Since the complexity of this assertion is $\Pi^1_2$, it is therefore absolute to $V$ by the Shoenfield
absoluteness theorem.
\end{proof}

Theorems \ref{Theorem.EveryCTMisInsideOmegaModelOfV=L} and \ref{Theorem.EveryCTSinsidePDMofV=L} admit of some striking examples. Suppose for instance that $0^\sharp$ exists.
Considering it as a real, the argument shows that $0^\sharp$ exists inside a pointwise definable model of $\ZFC+{V{=}L}$, well-founded far beyond $\omega_1^L$. So we achieve the
bizarre situation in which the true $0^\sharp$ sits unrecognized, yet definable, inside a model of $V=L$ which is well-founded a long way. For a second example, consider a forcing
extension $V[g]$ by the forcing to collapse $\omega_1$ to $\omega$. The generic filter $g$ is coded by a real, and so in $V[g]$ there is a model $M\satisfies\ZFC+{V{=}L}$ with $g\in
M$ and $M$ well-founded beyond $\omega_1^V$. The model $M$ believes that the generic object $g$ is actually constructible, constructed at some
(necessarily nonstandard) stage beyond $\omega_1^V$. Surely these models are unusual.

The theme of these arguments goes back, of course, to an elegant theorem of Barwise, theorem \ref{Theorem.BarwiseEndExtensionToV=L}, asserting that every countable model of \ZF\ has
an end-extension to a model of $\ZFC+{V{=}L}$. In Barwise's theorem, the original model becomes merely a subset of the end-extension, rather than an element of the end-extension as
in theorems \ref{Theorem.EveryCTMisInsideOmegaModelOfV=L} and \ref{Theorem.EveryCTSinsidePDMofV=L}. By giving up on the goal of making the original universe itself a set in the
end-extension, Barwise seeks only to make the elements of the original universe constructible in the extension, and is thereby able to achieve the full theory of $\ZFC+{V{=}L}$ in
the end-extension, without the extra hypothesis as in theorem \ref{Theorem.EveryCTSinsidePDMofV=L}, which cannot be omitted there. Another important difference is that Barwise's
theorem \ref{Theorem.BarwiseEndExtensionToV=L} also applies to nonstandard models.

\begin{theorem}\label{Theorem.BarwiseEndExtensionToV=L} {\rm(Barwise \cite{Barwise1971:InfinitaryMethodsInTheModelTheoryOfSetTheoryLC69})}
 Every countable model of \ZF\ has an end-extension to a model of $\ZFC+{V{=}L}$.
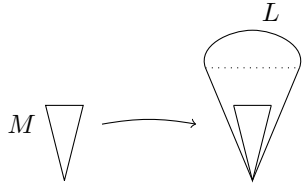
$$\begin{tikzpicture}[scale=.25]
 \draw (-10,0) --(-9,4) --(-11,4) --(-10,0);
 \draw (-1,4) --(0,0) --(1,4);
 \draw (1,4) --(-1,4);
 \draw[->] (-8,3) to [out=10,in=170](-3,3);
 \draw (-2.5,6) --(0,0) --(2.5,6);
 \draw (2.5,6) to [out=70,in=0] (0,8) to [out=180,in=110](-2.5,6);
 \draw[dotted] (-2.5,6) --(2.5,6);
 \node[anchor=north east] at (-11,4) {$M$};
 \node[anchor=south west] at (0,8) {$L$};
\end{tikzpicture}$$
\end{theorem}

Let me briefly outline a proof in the case of a countable transitive model $M\satisfies\ZF$. For such an $M$, let $T$ be the theory \ZFC\ plus the infinitary assertions
$\sigma_a=\forall z\,(z\in \check a\iff\bigvee_{b\in a}z=\check b)$, for every $a\in M$, in the $L_{\omega_1,\omega}$ language of set theory with constant symbol $\check a$ for every element $a\in M$. The $\sigma_a$ assertions, which are expressible in $L_{\infty,\omega}$ logic in the sense of $M$, ensure that the models of $T$ are precisely (up to isomorphism)
the end-extensions of $M$ satisfying \ZFC. What we seek, therefore, is a model of the theory $T+{V{=}L}$. Suppose toward contradiction that there is none. I claim consequently that
there is a proof of a contradiction from $T+{V{=}L}$ in the infinitary deduction system for $L_{\infty,\omega}$ logic, with such infinitary rules as: from $\sigma_i$ for $i\in I$,
deduce $\bigwedge_i \sigma_i$. Furthermore, I claim that there is such a proof inside $M$. Suppose not. Then $M$ thinks that the theory $T+{V{=}L}$ is consistent in
$L_{\infty,\omega}$ logic. We may therefore carry out a Henkin construction over $M$ by building a new theory $T^+\of M$ extending $T+{V{=}L}$, with infinitely many new constant
symbols, adding one new sentence at a time, each involving only finitely many of the new constants, in such a way so as to ensure that (i) the extension at each stage remains
$M$-consistent; (ii) $T^+$ eventually includes any given $L_{\infty,\omega}$ sentence in $M$ or its negation, for sentences involving only finitely many of the new constants; (iii)
$T^+$ has the Henkin property in that it contains $\exists x\, \varphi(x,\vec c)\implies\varphi(d,\vec c)$, where $d$ is a new constant symbol used expressly for this formula; and
(iv) whenever a disjunct $\bigvee_i\sigma_i$ is in $T^+$, then also some particular $\sigma_i$ is in $T^+$. We may build such a $T^+$ in $\omega$ many steps just as in the classical
Henkin construction. If $N$ is the Henkin model derived from $T^+$, then an inductive argument shows that $N$ satisfies every sentence in $T^+$, and in particular, it is a model
of $T+{V{=}L}$, which contradicts our assumption that this theory had no model. So there must be a proof of a contradiction from $T+{V{=}L}$ in the deductive system for
$L_{\infty,\omega}$ logic inside $M$. Since the assertion that there is such a proof is $\Sigma_1$ assertion in the language of set theory, it follows by the Levy-Shoenfield theorem
(corollary \ref{Corollary.LandVSameSigma_1theory}) that there is such a proof inside $L^M$, and indeed, inside $L_{\omega_1}^M$. This proof is a countable object in $L^M$ and uses
the axioms $\sigma_a$ only for $a\in L_{\omega_1}^M$. But $L^M$ satisfies the theory $T+{V{=}L}$ and also $\sigma_a$ for all such $a$ and hence is a model of the theory from which we
had derived a contradiction. This violates soundness for the deduction system, and so $T+{V{=}L}$ has a model after all. Consequently, $M$ has an end-extension satisfying
$\ZFC+{V{=}L}$, as desired, and this completes the proof.

We may attain a stronger theorem, where every $a\in M$ becomes countable in the end-extension model, simply by adding the assertions `{\it $\check a$ is countable}' to the theory
$T$. The point is that ultimately the proof of a contradiction exists inside $L_{\omega_1}^M$, and so the model $L^M$ satisfies these additional assertions for the relevant $a$.
Similarly, we may also arrange that the end-extension model is pointwise definable, meaning that every element in it is definable without parameters. This is accomplished by adding
to $T$ the infinitary assertions $\forall z\,\bigvee_{\varphi}\forall x (\varphi(x)\iff x=z)$, taking the disjunct over all first-order formulas $\varphi$. These assertions ensure
that every $z$ is defined by a first-order formula, and the point is that the $\sigma_a$ arising in the proof can be taken not only from $L^M$, but also from amongst the definable
elements of $L^M$, since these constitute an elementary substructure of $L^M$.

Remarkably, the theorem is true even for nonstandard models $\mathcal{M}$, but the proof above requires modification, since the infinitary deductions of $M$ may not be well-founded
deductions, and this prevents the use of soundness to achieve the final contradiction. (One can internalize the contradiction to soundness, if $M$ should happen to have an
uncountable $L_\beta\satisfies\ZFC$, or even merely arbitrarily large such $\beta$ below $(\omega_1^L)^M$.) To achieve the general case, however, Barwise uses his compactness theorem \cite{Barwise1969:InfinitaryLogicAndAdmissibleSetsJSL} and the theory of admissible covers to replace the ill-founded model $\mathcal{M}$ with a closely related admissible set in
which one may find the desired well-founded deductions and ultimately carry out an essentially similar argument. I refer the reader to the accounts in
\cite{Barwise1971:InfinitaryMethodsInTheModelTheoryOfSetTheoryLC69} and in \cite[appendix]{Barwise1975:AdmissibleSetsStructures}.

It turns out, however, that one does not need this extra technology in the case of an $\omega$-nonstandard model $\mathcal{M}$ of \ZF, and so let me explain this case. Suppose that
$\mathcal{M}=\<M,{\in^{\mathcal{M}}}>$ is an $\omega$-nonstandard model of \ZF. Let $T$ again be the theory $\ZFC+\sigma_a$ for $a\in M$, where again $\sigma_a=\forall z\,(z\in
\check a\iff\bigvee_{b\in^{\mathcal{M}}a}z=\check b)$. Suppose there is no model of $T+{V{=}L}$. Consider the nonstandard theory $\ZFC^M$, which includes many nonstandard formulas.
By the reflection theorem, every finite collection of \ZFC\ axioms is true in arbitrarily large $L_\beta^{\mathcal{M}}$, and so by overspill there must be a nonstandard finite theory $\ZFC^*$ in $\mathcal{M}$ that includes every standard \ZFC\ axiom and which $\mathcal{M}$ believes to hold in some $L_\beta^{\mathcal{M}}$ for some uncountable ordinal $\beta$ in
$\mathcal{M}$. Let $T^*$ be the theory $\ZFC^*$ plus all the $\sigma_a$ for $a\in M$. This theory is $\Sigma_1$ definable in $\mathcal{M}$, and I claim that $\mathcal{M}$ must have a proof of a contradiction from $T^*+{V{=}L}$ in the infinitary logic $L_{\infty,\omega}^{\mathcal{M}}$. If not, then the same Henkin construction as above still works, working with
nonstandard formulas inside $\mathcal{M}$, and the corresponding Henkin model satisfies all the actual (well-founded) assertions in $T^*+{V{=}L}$, which includes all of $T+{V{=}L}$,
contradicting our preliminary assumption. So $\mathcal{M}$ has a proof of a contradiction from $T^*+{V{=}L}$. Since the assertion that there is such a proof is $\Sigma_1$, we again
find a proof in $L^M$ and even in $L_{\omega_1}^M$. But we may now appeal to the fact that $M$ thinks $L_\beta^M$ is a model of $\ZFC^*$ plus $\sigma_a$ for every $a\in
L_{\omega_1}^M$, which contradicts the soundness principle of the infinitary deduction system {\it inside} $M$. The point is that even though the deduction is nonstandard, this
doesn't matter since we are applying soundness not externally but inside $\mathcal{M}$. The contradiction shows that $T+{V{=}L}$ must have a model after all, and so $\mathcal{M}$ has an end-extension satisfying $\ZFC+{V{=}L}$, as desired. Furthermore, we may also ensure that every element of $\mathcal{M}$ becomes countable in the end-extension as before.

Let me conclude this section by mentioning another sense in which every countable model of set theory is compatible in principle with $V=L$.

\begin{theorem} (Hamkins \cite{Hamkins:EveryCountableModelOfSetTheoryEmbedsIntoItsOwnL})
 Every countable model of set theory $\<M,{\in^M}>$, including every transitive model, is isomorphic to a submodel of its own constructible universe $\<L^M,{\in^M}>$. In other words,  there is an embedding $j:M\to L^M$, which is elementary for quantifier-free assertions.
$$
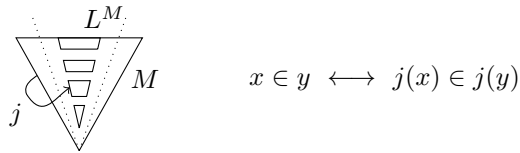
\hbox{\begin{tikzpicture}[xscale=.07,yscale=.3]
 \draw (-0,0) --(12,5) --(-12,5) --(0,0);
 \draw[dotted] (0,0) --(9,6);
 \draw[dotted] (0,0) --(-9,6);
 \node[anchor=south west] at (-1,5) {$L^M$};
 \draw (0,1) --(1,2) --(-1,2) --(0,1);
 \draw (1.4,2.4) --(2.1,3.1) --(-2.1,3.1) --(-1.4,2.4) --(1.4,2.4);
 \draw (2.5,3.5) --(3,4) --(-3,4) --(-2.5,3.5) --(2.5,3.5);
 \draw (3.5,4.5) --(4,5) --(-4,5) --(-3.5,4.5) --(3.5,4.5);
 \draw[->] (-8,3.33) to [out=190,in=150] (-10,2.5) to [out=-20,in=190] (-1.8,2.8);
 \node[anchor=north east] at (-9,2.5) {$j$};
 \node[anchor=north west] at (8,4) {$M$};
\end{tikzpicture}}\qquad\quad\raise 25pt\hbox{$x\in y\ \longleftrightarrow\ j(x)\in j(y)$}$$
\end{theorem}

\noindent Another way to say this is that every countable model of set theory is a submodel of a model isomorphic to $L^M$. If we lived inside $M$, then by adding new sets and elements, our
universe could be transformed into a copy of the constructible universe $L^M$.

\section{An upwardly extensible concept of set}\label{Section.UpwardlyExtensibleConceptOfSet}

I would like now to explain how the mathematical facts identified in the previous section weaken support for the $V\neq L$ via maximize position, particularly for those set theorists
inclined toward a pluralist or multiverse conception of the subject.

To my way of thinking, theorem \ref{Theorem.VandLhaveSameTransitiveModels} already provides serious resistance to the $V\neq L$ via maximize argument, even without the multiverse
ideas I shall subsequently discuss. The point is simply that much of the force and content of large cardinal set theory, presumed lost under $V=L$, is nevertheless still provided
when the large cardinal theory is undertaken merely with countable transitive models, and theorem \ref{Theorem.VandLhaveSameTransitiveModels} shows that this can be done while
retaining $V=L$. We often regard a large cardinal argument or construction as important---such as Baumgartner's forcing of \PFA\ over a model with a supercompact cardinal---because
it helps us to understand a greater range for set-theoretic possibility. The fact that there is indeed an enormous range of set-theoretic possibility is the central discovery of the last half-century of set theory, and one wants a philosophical account of the phenomenon. The large cardinal arguments enlarge us by revealing the set-theoretic situations to
which we might aspire. Because of the Baumgartner argument, for example, we may freely assert $\ZFC+\PFA$ with the same gusto and confidence that we had for $\ZFC$ plus a
supercompact cardinal, and furthermore we gain detailed knowledge about how to transform a universe of the latter theory to one of the former and how these worlds are
related.\footnote{The converse question, however, whether we may transform models of \PFA\ to models of $\ZFC+\exists$ supercompact cardinal, remains open. Many set theorists have
conjectured that these theories are equiconsistent.} Modifications of that construction are what led us to worlds where \MM\ holds and $\MM^+$ and so on. From this perspective, a
large part of the value of large cardinal argument is supplied already by our ability to carry it out over a transitive model of \ZFC, rather than over the full universe $V$.

The observation that we gain genuine set-theoretic insights when working merely over countable transitive models is reinforced by the fact that the move to countable transitive
models is or at least was, for many set theorists, a traditional part of the official procedure by which the forcing technique was formalized. (Perhaps a more common contemporary
view is that this is an unnecessary pedagogical simplification, for one can formalize forcing over $V$ internally as a \ZFC\ construction.) Another supporting example is provided by
the inner model hypothesis of \cite{Friedman2006:InternalConsistencyAndIMH}, a maximality-type principle whose very formalization seems to require one to think of the universe as a
toy model, for the axiom is stated about $V$ as it exists as a countable transitive model in a larger universe. In short, much of what we hope to achieve with our strong set theories is already achieved merely by having transitive models of those theories, and theorem \ref{Theorem.VandLhaveSameTransitiveModels} shows that the existence of any and all such kind of transitive models is fully and equally consistent with our retaining $V=L$. Because of this, the $V\neq L$ via maximize argument begins to lose its force.

Nearly every set theorist entertaining some strong set-theoretic hypothesis $\psi$ is generally also willing to entertain the hypothesis that $\ZFC+\psi$ holds in a transitive model. To
be sure, the move from a hypothesis $\psi$ to the assertion `there is a transitive model of $\ZFC+\psi$' is strictly increasing in consistency strength, a definite step up, but a
small step. Just as philosophical logicians have often discussed the general principle that if you are willing to assert a theory $T$, then you are also or should also be willing to
assert that `$T$ is consistent,' in set theory we have the similar principle, that if you are willing to assert $T$, then you are or should be willing to assert that `there is a
transitive model of $T$'. What is more, such a principle amounts essentially to the mathematical content of the philosophical reflection arguments, such as in
\cite{Reinhardt1974:RemarksOnReflectionPrinciplesLargeCardinalsAndElementaryEmbeddings}, that are often used to justify large cardinal axioms. As a result, one has a kind of
translation that maps any strong set-theoretic hypothesis $\psi$ to an assertion `there is a transitive model of $\ZFC+\psi$', which has the same explanatory force in terms of
describing the range of set-theoretic possibility, but which because of the theorems of section \ref{Section.V=LCompatibleWithStrength} remains compatible with $V=L$.

This perspective appears to rebut Steel's claims, mentioned in the opening section of this article, that ``there is no translation'' from the large cardinal realm to the $V=L$
context and that ``adding $V=L$...prevents us from asking as many questions.'' Namely, the believer in $V=L$ seems fully able to converse meaningfully with any large cardinal set
theorist, simply by imagining that the large cardinal set theorist is currently living inside a countable transitive model. By applying the translation
 $$\psi\quad\longmapsto\quad\text{`there is a transitive model of $\ZFC+\psi$'},$$
the $V=L$ believer steps up in strength above the large cardinal set theorist, while retaining $V=L$ and while remaining fully able to analyze and carry out the large cardinal set
theorist's arguments and constructions inside that transitive model. Furthermore, if the large cardinal set theorist believes in her axiom because of the philosophical reflection principle arguments, then she agrees that set-theoretic truth is ultimately captured inside transitive sets, and so ultimately she agrees with the step up that the $V=L$ believer made, to put the large cardinal theory inside a transitive set. This simply reinforces the accuracy with which the $V=L$ believer has captured the situation.

Although the translation I am discussing is not a `fair interpretation' in the technical sense of \cite{Maddy1998:V=LAndMaximize}, as discussed in section \ref{Section.CriticismOfMaddy}, nevertheless it does seem to me to be a fair interpretation in a sense that matters, because it allows the $V=L$ believer to understand and appreciate the large cardinal set theorist's arguments and constructions.

Let me now go a bit further. My claim is that on the multiverse view as I describe it in \cite{Hamkins2012:TheSet-TheoreticalMultiverse} (see also
\cite{GitmanHamkins2010:NaturalModelOfMultiverseAxioms,Hamkins:TheMultiverse:ANaturalContext,Hamkins:IsTheDreamSolutionToTheContinuumHypothesisAttainable}), the nature of the full
outer multiverse of $V$ is revealed in part by the toy simulacrum of it that we find amongst the countable models of set theory. For all we know, our current set-theoretic universe
$V$ is merely a countable transitive set inside another much larger universe $V^+$, which looks upon $V$ as a mere toy. And so when we can prove that a certain behavior is pervasive
in the toy multiverse of any model of set theory, then we should expect to find this behavior also in the toy multiverse of $V^+$, which includes a meaningfully large part of the actual multiverse of $V$. In this way, we come to learn about the full multiverse of $V$ by undertaking a general study of the toy model multiverses. Just as every countable model has actual forcing extensions, we expect our full universe to have actual forcing extensions; just as every countable model can be end-extended to a model of $V=L$, we expect the full universe $V$ can be end-extended to a universe in which $V=L$ holds; and so on. How fortunate it is that the study of the connections between the countable models of set theory is a purely mathematical activity that can be carried out within our theory. This mathematical knowledge, such as the results mentioned in section \ref{Section.V=LCompatibleWithStrength} or the results of \cite{GitmanHamkins2010:NaturalModelOfMultiverseAxioms}, which show that the multiverse axioms of \cite{Hamkins2012:TheSet-TheoreticalMultiverse} are true amongst the countable computably-saturated models of set theory, in turn supports philosophical conclusions about the nature of the full set-theoretic multiverse.

The principle that pervasive features of the toy multiverses are evidence for the truth of those features in the full multiverse is a reflection principle similar in kind to those
that are often used to provide philosophical justification for large cardinals. Just as those reflection principles regard the full universe $V$ as fundamentally inaccessible, yet
reflected in various much smaller pieces of the universe, the principle here regards the full multiverse as fundamentally inaccessible, yet appearing in part locally as a toy
multiverse within a given universe. So our knowledge of what happens in the toy multiverses becomes evidence of what the full multiverse may be like.

Ultimately, the multiverse vision entails an upwardly extensible concept of set, where any current set-theoretic universe may be extended to a much larger, taller universe. The
current universe becomes a countable model inside a larger universe, which has still larger extensions, some with large cardinals, some without, some with the continuum hypothesis,
some without, some with $V=L$ and some without, in a series of further extensions continuing longer than we can imagine. Models that seem to have $0^\sharp$ are extended to larger
models where that version of $0^\sharp$ no longer works as $0^\sharp$, in light of the new ordinals. Any given set-theoretic situation is seen as fundamentally compatible with $V=L$, if one is willing to make the move to a better, taller universe. Every set, every universe of sets, becomes both countable and constructible, if we wait long enough. Thus, the
constructible universe $L$ becomes a {\it rewarder of the patient}, revealing hidden constructibility structure for any given mathematical object or universe, if one should only
extend the ordinals far enough beyond one's current set-theoretic universe. This perspective turns the $V\neq L$ via maximize argument on its head, for by maximizing the
ordinals, we seem able to recover $V=L$ as often as we like, extending our current universe to larger and taller universes in diverse ways, attaining $V=L$ and destroying it in an
on-again, off-again pattern, upward densely in the set-theoretic multiverse, as the ordinals build eternally upward, eventually exceeding any particular conception of them.


\begin{thebibliography}{00}    
\bibitem{Barwise1969:InfinitaryLogicAndAdmissibleSetsJSL}
Jon Barwise.
\newblock Infinitary logic and admissible sets.
\newblock {\em J. Symbolic Logic}, 34(2):226--252, 1969.

\bibitem{Barwise1971:InfinitaryMethodsInTheModelTheoryOfSetTheoryLC69}
Jon Barwise.
\newblock Infinitary methods in the model theory of set theory.
\newblock In {\em Logic {C}olloquium '69 ({P}roc. {S}ummer {S}chool and
  {C}olloq., {M}anchester, 1969)}, pages 53--66. North-Holland, Amsterdam,
  1971.

\bibitem{Barwise1975:AdmissibleSetsStructures}
Jon Barwise.
\newblock {\em Admissible sets and structures}.
\newblock Springer-Verlag, Berlin, 1975.
\newblock An approach to definability theory, Perspectives in Mathematical
  Logic.

\bibitem{Friedman2006:InternalConsistencyAndIMH}
Sy-David Friedman.
\newblock Internal consistency and the inner model hypothesis.
\newblock {\em Bull. Symbolic Logic}, 12(4):591--600, 2006.

\bibitem{GitmanHamkins2010:NaturalModelOfMultiverseAxioms}
Victoria Gitman and Joel~David Hamkins.
\newblock A natural model of the multiverse axioms.
\newblock {\em Notre Dame J. Form. Log.}, 51(4):475--484, 2010.

\bibitem{Hamkins:EveryCountableModelOfSetTheoryEmbedsIntoItsOwnL}
Joel~David Hamkins.
\newblock Every countable model of set theory embeds into its own constructible
  universe.
\newblock pages 1--26.
\newblock under review.

\bibitem{Hamkins:IsTheDreamSolutionToTheContinuumHypothesisAttainable}
Joel~David Hamkins.
\newblock Is the dream solution of the continuum hypothesis attainable?
\newblock pages 1--10.
\newblock under review.

\bibitem{Hamkins:TheMultiverse:ANaturalContext}
Joel~David Hamkins.
\newblock The set-theoretic multiverse : A natural context for set theory.
\newblock {\em Annals of the Japan Association for Philosophy of Science},
  19:37--55, 2011.

\bibitem{Hamkins2012:TheSet-TheoreticalMultiverse}
Joel~David Hamkins.
\newblock The set-theoretical multiverse.
\newblock {\em Review of Symbolic Logic}, 5:416--449, 2012.

\bibitem{HamkinsLinetskyReitz:PointwiseDefinableModelsOfSetTheory}
Joel~David Hamkins, David Linetsky, and Jonas Reitz.
\newblock Pointwise definable models of set theory.
\newblock {\em to appear in Journal of Symbolic Logic}.

\bibitem{Isaacson2008:TheRealityOfMathematicsAndTheCaseOfSetTheory}
Daniel Isaacson.
\newblock The reality of mathematics and the case of set theory.
\newblock In Zsolt Novak and Andras Simonyi, editors, {\em Truth, Reference,
  and Realism}. Central European University Press, 2008.

\bibitem{Loewe2001:AFirstGlanceAtNonrestrictiveness}
Benedikt L\"owe.
\newblock A first glance at non-restrictiveness.
\newblock {\em Philosophia Mathematica}, 9(3):347--354, 2001.

\bibitem{Loewe2003:ASecondGlanceAtNonrestrictiveness}
Benedikt L\"owe.
\newblock A second glance at non-restrictiveness.
\newblock {\em Philosophia Mathematica}, 11(3):323--331, 2003.

\bibitem{Maddy1988:BelievingTheAxiomsI}
Penelope Maddy.
\newblock Believing the axioms, {I}.
\newblock {\em The Journal of Symbolic Logic}, 53(2):481--511, 1988.

\bibitem{Maddy1988:BelievingTheAxiomsII}
Penelope Maddy.
\newblock Believing the axioms, {II}.
\newblock {\em The Journal of Symbolic Logic}, 53(3):736--764, 1988.

\bibitem{Maddy1998:V=LAndMaximize}
Penelope Maddy.
\newblock {$V=L$} and {MAXIMIZE}.
\newblock In {\em Logic {C}olloquium '95 ({H}aifa)}, volume~11 of {\em Lecture
  Notes Logic}, pages 134--152. Springer, Berlin, 1998.

\bibitem{Martin2001:MultipleUniversesOfSetsAndIndeterminateTruthValues}
Donald~A. Martin.
\newblock Multiple universes of sets and indeterminate truth values.
\newblock {\em Topoi}, 20(1):5--16, 2001.

\bibitem{Reinhardt1974:RemarksOnReflectionPrinciplesLargeCardinalsAndElementaryEmbeddings}
W.~N. Reinhardt.
\newblock Remarks on reflection principles, large cardinals, and elementary
  embeddings.
\newblock {\em Proceedings of Symposia in Pure Mathematics}, 13(II):189--205,
  1974.

\bibitem{Shelah2003:LogicalDreams}
Saharon Shelah.
\newblock Logical dreams.
\newblock {\em Bulletin of the American Mathematical Society}, 40:203--228,
  2003.

\bibitem{Steel2004:SlidesGenericAbsolutenessAndTheContinuumProblem}
John~R. Steel.
\newblock Generic absoluteness and the continuum problem.
\newblock Slides for a talk at the Laguna workshop on philosophy and the
  continuum problem (P. Maddy and D. Malament organizers) March, 2004.
\newblock http://math.berkeley.edu/~steel/talks/laguna.ps.

\end{thebibliography}

\end{document}